\documentclass[11pt,a4paper]{article}

\pdfoutput=1

\usepackage{graphics,std-defs}

\def\procZ#1{\{#1\}_{\scriptscriptstyle n\in\ZZ}}
\def\procZk#1{\{#1\}_{\scriptscriptstyle k\in\ZZ}}
 
\def\theenumi{{\rm({\it \roman{enumi}\/})}}
\def\intsurt{\int_0^{T_1}\!}
\def\DW{{\Delta W_{\theta,h}}}
\def\intDW{\int_0^\DW\mskip-15mu}
\def\intmDW{\int_0^{-\DW}\mskip-15mu}
\def\DWh{{\Delta W_{\theta+h,h}}}
\def\intDWh{\int_0^\DWh\mskip-15mu}
\def\intmDWh{\int_0^{-\DWh}\mskip-15mu}
\def\DDW{{\Delta^2 W_{\theta,h}}}

\begin{document}

\title{Stationary IPA Estimates for Non-Smooth G/G/1/$\infty$
       Functionals via Palm Inversion and Level-Crossing Analysis.
       \thanks{This article has been presented at the $31^{st}$ IEEE CDC,
               Dec.\ 16-18 1992, Tucson, Arizona, USA.}}

\author{
 Pierre Br\'emaud\\
  Laboratoire des Signaux et Syst\`emes, CNRS
  \thanks{Laboratoire des Signaux et Syst\`emes, CNRS - ESE, Plateau 
   du Moulon, 91190 Gif-sur-Yvette, France}
\and
 Jean-Marc Lasgouttes\\
  INRIA
  \thanks{INRIA, Domaine de Voluceau, Rocquencourt, B.P. 105, 78153
   Le Chesnay Cedex, France}}
\date{April 1992; revised January 1993, September 1993}

\maketitle

\begin{abstract}
We give stationary estimates for the derivative of the expectation of
a non-smooth function of bounded variation $f$ of the workload in a
G/G/1/$\infty$ queue, with respect to a parameter influencing the
distribution of the input process. For this, we use an idea of
Konstantopoulos and Zazanis \cite{KonZaz:1} based on the Palm
inversion formula, however avoiding a limiting argument by performing
the level-crossing analysis thereof globally, via Fubini's theorem.
This method of proof allows to treat the case where the workload
distribution has a mass at discontinuities of $f$ and where the
formula of \cite{KonZaz:1} has to be modified. The case where the
parameter is the speed of service or/and the time scale factor of the
input process is also treated using the same approach.
\end{abstract}

\section{Introduction.}\label{sec:introduction}

Consider a stationary G/G/1/$\infty$ queue in which customers arrive
according to a stationary process $\procZ{T_n}$. The customer $n$ asks
for a service time $\sigma_n(\theta)$, where $\theta$ is a real
parameter in the compact interval $\Theta$ and
$\procZ{\sigma_n(\theta)}$ is an i.i.d\ sequence. Let $\procZ{\tau_n}$
denote the inter-arrival times process satisfying
$\tau_n=T_{n+1}-T_n$. Assume that the queue is stationary and let
$W_\theta(t)$ be the remaining work in the system at time $t$---see
Figure \ref{fig-w}--- given by Lindley's equation
\begin{figure}
\centering
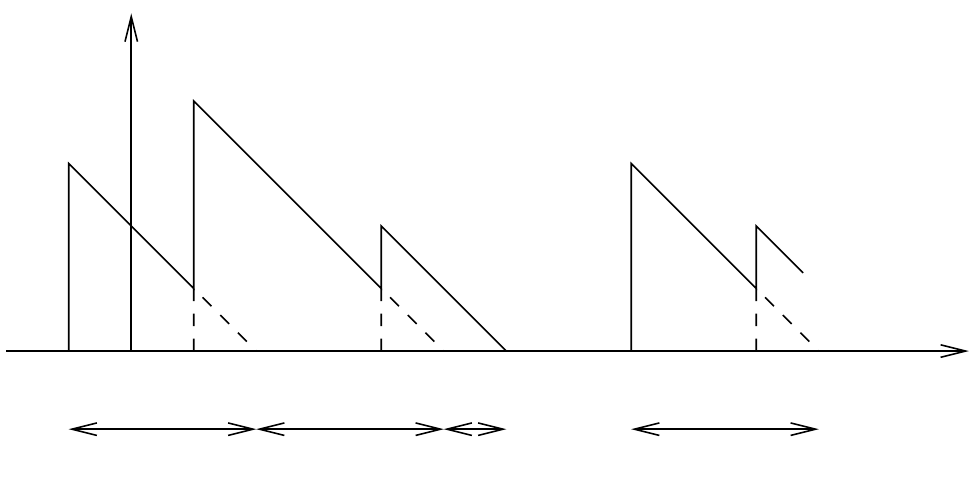
\caption{workload of a G/G/1 queue.}
\label{fig-w}
\end{figure}
\begin{equation}\label{eq:lindley}
W_\theta(t)=\Bigl(W_\theta(T_n-)+\sigma_n(\theta)-(t-T_n)\Bigr)^+,\
t\in[T_n,T_{n+1}),
\end{equation}
with the notation $x^+\egaldef \max(x,0)$. Given a real function
$f$, consider the functional $J(\theta)$ defined as
$$
J(\theta)\egaldef\EE f(W_\theta(0)).
$$

We want to estimate, if it exists, the derivative of $J$ with respect
to $\theta$. To this end, we use Infinitesimal Perturbation Analysis
(IPA), a method first introduced by Ho and Cao~\cite{HoCao:2} and
further developed by Cao~\cite{Cao:1}, Suri and
Zazanis~\cite{SurZaz:1} and recently Konstantopoulos and
Zazanis~\cite{KonZaz:1}. Glasserman~\cite{Gla:1} and Ho and
Cao~\cite{HoCao:1} summarize and review most previous results on IPA.
Alternative methods have been used to estimate derivatives, namely
Smooth Perturbation Analysis (SPA, see Suri and
Zazanis~\cite{ZazSur:3}, Gong and Ho~\cite{GonHo:1}, Glasserman and
Gong~\cite{GlaGon:1}, Fu and Hu~\cite{FuHu:2}), Likelihood Ratio
Method (LRM, see e.g.\ Reiman and Weiss~\cite{ReiWei:1} or
Glynn~\cite{Gly:1}) and Rare Perturbation Analysis (RPA, see Br\'emaud
and V\'azquez-Abad~\cite{BreVaz:1} and Br\'emaud~\cite{Bre:1}).

In this article, we aim to prove that, under appropriate conditions
\begin{eqnarray}\label{eq:baseIPA}
\longeqn\lim_{h\to0}{1\over h}\Bigl[\EE f(W_{\theta+h}(0))-\EE f(W_\theta(0))\Bigr]\\
&=& \EE\lim_{h\to0}{1\over h}\Bigl[f(W_{\theta+h}(0))- f(W_\theta(0))\Bigr]
= \EE \d\theta{}f(W_\theta(0))\nonumber
\end{eqnarray}
and we give a formula replacing~(\ref{eq:baseIPA}) when $f$ is not
differentiable but is of bounded variation. This formula was obtained
by Konstantopoulos and Zazanis~\cite{KonZaz:1} under stronger
assumptions on the service times distributions. However, due to the
difficulty of passing to the limit in their approximation procedure,
their formula does not give any insight on the equality of the
left-hand and right-hand derivatives ; this information is crucial for
practical use of the derivative estimator. Our method of proof avoids
the passage to the limit and therefore allows for better control of
the computations.  Moreover, it can be extended in many ways to handle
different situations.

The article is organized as follow: in Section~\ref{sec:construction},
we give a construction of the G/G/1 queue and we derive some basic
properties. The main result of the article is given in
Section~\ref{sec:service} and the same method is applied to
second-order derivatives in Section~\ref{sec:2ndorder};
Section~\ref{sec:other} shows how our method can be extended to other
parameters, respectively the speed of the server and the rate of
arrival in the system. Section~\ref{sec:estimate} discusses the
implementation of the estimates and a short review of Palm
probabilities can be found in the appendix.

\section{Construction of the G/G/1 queue.}\label{sec:construction}

In a formula like (\ref{eq:baseIPA}), the probability space does not
depend on $\theta$. To obtain this independence, we use the inversion
representation (see Suri~\cite{Sur:3}) to generate service times: let
$\procZ{\xi_n}$ be a sequence of random variables uniformly
distributed on $[0,1]$. Let $F(\cdot,\theta)$ be the common
distribution function of service times; we can define its inverse
function
$$
G(\xi,\theta)=\sup \bigl(x\geq0\ :\
F(x,\theta)\leq\xi\bigr).
$$

Then $\sigma_n(\theta)\egaldef G(\xi_n,\theta)$ is distributed
according to $F(\cdot,\theta)$. This means that, if we choose as basic
stationary random sequences $\procZ{\tau_n}$ and $\procZ{\xi_n}$, we
define the queue on a probability space independent from $\theta$. We
note $\lambda$ the intensity of the input process and $\PP^0$ the
associated Palm probability---see Appendix for notations and details.
In order to apply IPA, the following assumption on service times is
needed:

\begin{ass}\label{ass:service}
The distribution of service times verifies the following conditions:
\begin{enumerate}
\item $\theta\mapsto G(\xi,\theta)$ is differentiable and
Lipschitz, that is
$$
\bigl|G(\xi,\theta_1)-G(\xi,\theta_2)\bigr|\leq
K^\sigma(\xi)|\theta_1-\theta_2|,
 \ \forall\,\theta_1,\theta_2\in\Theta;
$$
\item $\lambda\EE^0\sigma^*_0<1$, with the notation 
$\sigma^*_n\egaldef\sup_{\theta\in\Theta}\sigma_n(\theta)
=\sup_{\theta\in\Theta}G(\xi_n,\theta)$.
\end{enumerate}
\end{ass}

Condition \ref{ass:service}-\romi\ ensures that we have enough
smoothness with respect to $\theta$ in the distribution of the service
times.  However, in a number of cases, $\xi_n$ will not be directly
known, in particular when observing a real experiment; this difficulty
can be overcome with the following classical proposition (
Suri~\cite{Sur:3}; for this formulation see Glasserman~\cite{Gla:1}):

\begin{envthm}{Proposition}\it
Suppose that \romi\ $F(\cdot,\theta)$ has a density $\partial_x
F(\cdot,\theta)$ which is strictly positive on an open interval
$I_\theta$ and zero elsewhere; and \romii\ $F$ is continuously
differentiable on $I_\theta\times\Theta$. Then
$$
\sigma'(\theta)=-{\partial_\theta F(\sigma(\theta),\theta)
                 \over \partial_x F(\sigma(\theta),\theta)}.
$$
\end{envthm}

In the above formula, the prime denotes the derivative with respect to
$\theta$. A case of particular interest is when $\theta$ is a scale
parameter of the service times, that is when
$\sigma(\theta)=\theta\eta$ for some random variable $\eta$. Then we
have directly
$$
\sigma'(\theta)=\eta={\sigma(\theta)\over\theta}.
$$

In particular, we do not need to know the real distribution of service
times unless we actually want to simulate them. Note that
\ref{ass:service}-\romi\ is similar to assumption \romi\ of Section~1 in
Konstantopoulos and Zazanis~\cite{KonZaz:1}; it is the classical
assumption on smoothed distributions needed for IPA.

\begin{figure}
\centering
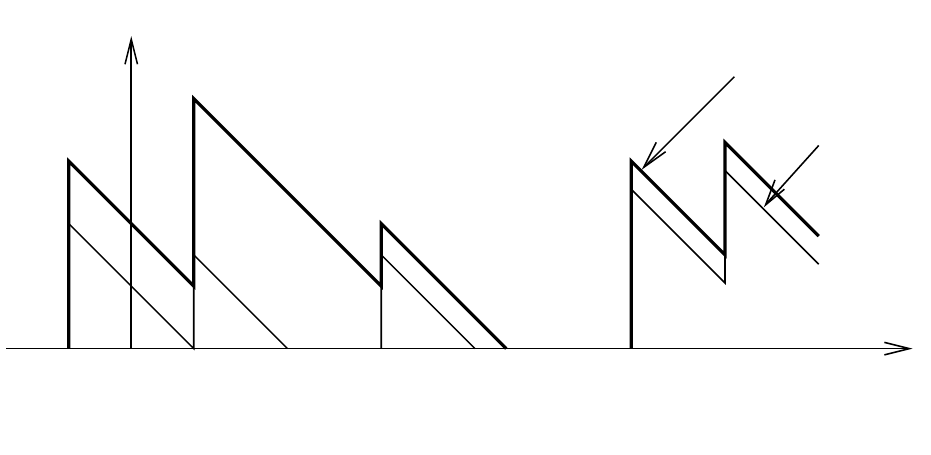
\caption{the domination property.}
\label{fig-domi}
\end{figure}
Using Assumption \ref{ass:service}-\romii, we derive a bound on the
size of the busy periods of the system for all possible values of
$\theta$.  We shall note that we don't know {\em a priori} whether
$\sigma^*_n$ is $\PP^0$-a.s.\ finite or not; however, this condition
is weaker than assumption \romii\ and \romiii\ of~\cite{KonZaz:1}.
With that in mind, let $\procZk{R^*_k}$ be the regeneration times at
which the arriving customers of the $*$-system find the queue
empty---the $*$-system is the queue with service times
$\procZ{\sigma^*_n}$, whereas the $\theta$-system uses service times
$\procZ{\sigma_n(\theta)}$. We can build the $\theta$-system from the
busy period process $\procZk{R_k^*}$ but with the service times given
by $\procZ{\sigma_n(\theta)}$, so that the following domination
property holds for the respective stationary workload of the
queues---see Figure~\ref{fig-domi}:
\begin{equation}\label{eq:domi}
W_\theta(t)\leq W^*(t),\ \forall\,\theta\in\Theta,\ \forall\ t\in\RR.
\end{equation}
 
With the above construction, we get
$\procZk{R^*_k}\subseteq\procZk{R_k(\theta)}$, where
$\procZk{R_k(\theta)}$---or simply $\procZk{R_k}$---denotes the
beginning of busy period process for the $\theta$-system. Moreover, we
have the boundary property
\begin{equation}\label{eq:bound}
R^*_-(t) \leq R_-(\theta)(t) \leq t < R_+(\theta)(t) \leq R_+^*(t).
\end{equation}

\section{An IPA estimator for general non-decreasing functions.}
\label{sec:service}

In this section, we show that IPA applies with any non-decreasing {\em
c\`adl\`ag} function $f$. But since $f$ is not required to be
continuous, we cannot apply (\ref{eq:baseIPA}) as such. First of all,
we need to introduce an assumption similar to assumptions {\bf A1},
{\bf A2} and {\bf A3$'$} of Konstantopoulos and
Zazanis~\cite{KonZaz:1}:

\begin{ass}\label{ass:tp}
The following inequalities hold:
\begin{enumerate}
\item $\EE^0[K^\sigma(\xi_0)]^4<\infty;$
\item $\EE^0[A([R_0^*,R_1^*))]^4<\infty;$
\item $\EE^0[f(W^*(0))]^2<\infty.$
\end{enumerate}
\end{ass}

\begin{thm}\label{thm:princ}
Let $\mu_f$ be the measure on $\RR$ associated with $f$. Assume
\ref{ass:service} and \ref{ass:tp} hold. Then $J$ admits a right
derivative with respect to $\theta$ given by
\begin{eqnarray}\label{eq:tp:Jr}
J'_r(\theta)
&=&\lambda\EE^0W'_\theta(0)\biggl[f(W_\theta(0))
                                 -f(W_\theta(T_1-))\nonumber\\
& &\mbox{\qquad}-\1_{\{W'_\theta(0)<0\}}
           \Bigl[\mu_f(\{W_\theta(0)\})
                -\mu_f(\{W_\theta(T_1-)\})\Bigr]\biggr],
\end{eqnarray}
and its left derivative is
\begin{eqnarray}\label{eq:tp:Jl}
J'_l(\theta)
&=&\lambda\EE^0W'_\theta(0)\biggl[f(W_\theta(0))
                                 -f(W_\theta(T_1-))\nonumber\\
& &\mbox{\qquad}-\1_{\{W'_\theta(0)>0\}}
          \Bigl[\mu_f(\{W_\theta(0)\})
               -\mu_f(\{W_\theta(T_1-)\})\Bigr]\biggr].
\end{eqnarray}
\end{thm}

\begin{example}\label{ex:proba}
With $f(w)=\1_{\{w\geq x\}}$, Theorem~\ref{thm:princ} yields
\begin{eqnarray*}
\d\theta{{}_r}\PP(W_\theta(0)>x)
&=&\lambda\EE^0W'_\theta(0)\biggl[
      \1_{(W_\theta(T_1-),W_\theta(0)]}(x)\\
& &\mbox{\qquad}-\1_{\{W'_\theta(0)<0\}}
           \Bigl[\1_{\{W_\theta(0)=x\}}
                -\1_{\{W_\theta(T_1-)=x\}}\Bigr]\biggl]\\
\d\theta{{}_l}\PP(W_\theta(0)>x)
&=&\lambda\EE^0W'_\theta(0)\biggl[
      \1_{(W_\theta(T_1-),W_\theta(0)]}(x)\\
& &\mbox{\qquad}-\1_{\{W'_\theta(0)>0\}}
           \Bigl[\1_{\{W_\theta(0)=x\}}
                -\1_{\{W_\theta(T_1-)=x\}}\Bigr]\biggl].\\
\end{eqnarray*}
\end{example}

Theorem \ref{thm:princ} shows that $J(\theta)$ admits right and left
derivatives even when $f$ is not continuous. But in a number of cases,
we can get the equality of these two derivatives:

\begin{cor}\label{cor:princ}
Assume \ref{ass:service} and \ref{ass:tp} hold. If $f$ is continuous
or if $W_\theta(0)$ and $W_\theta(T_1-)$ admit densities with respect
to $\PP^0$ then $J(\theta)$ is differentiable and 
\begin{equation}\label{eq:cp:Jprime}
J'(\theta)=\lambda\EE^0W'_\theta(0)[f(W_\theta(0))-f(W_\theta(T_1-))].
\end{equation}
\end{cor}
\begin{proof}{}
If $f$ is continuous, then $w\mapsto \mu_f(\{w\})\equiv 0$. If
$W_\theta(0)$ admits a $\PP^0$-density, say $\gamma^0(w)$, we can use
the fact that $\mu_f(\{\cdot\})=0$ almost everywhere for the Lebesgue
measure:

\begin{eqnarray*}
|\EE^0\1_{\{W'_\theta(0)<0\}}\mu_f(\{W_\theta(0)\})|
&\leq& \EE^0\mu_f(\{W_\theta(0)\})\\
&=&    \int_0^\infty \mu_f(\{w\})\gamma^0(w)\,dw=0.
\end{eqnarray*}

In either case, the result is proved.
\end{proof}

\begin{remark}
In the case where $f$ admits a derivative $f'$, we can use the
inversion formula (\ref{eq:inversion}) of Appendix and write
(\ref{eq:cp:Jprime}) as
$$
J'(\theta)=\lambda\EE^0[\int_0^{T_1}W'_\theta(t)f'(W_\theta(t))dt]
=\EE[W'_\theta(0)f'(W_\theta(0))],
$$
thus obtaining the expected IPA estimate (\ref{eq:baseIPA}). In this
computation, we used the fact that $W'_\theta(t)$ is constant between
arrivals during busy periods, and zero during idle periods. A
comparison between the two estimates is made in Section~\ref{sec:estimate}.
\end{remark}

Before starting the proof of the theorem, let us mention that our
derivation is different from Konstantopoulos and Zazanis
\cite{KonZaz:1} in two respects: first we do not require an
approximation procedure and we treat directly a non decreasing
function $f$. This is made possible by the simple crucial observation
that
$$
f(y)-f(x)=\int_{(x,y]}\mu_f(dz)\ \mbox{ for all }x\leq y,
$$
which allows us to have a better view of the residual terms in the
level crossing analysis that follows. The result can be applied to any
function of bounded variation if assumption \ref{ass:tp} is verified
by both the increasing and decreasing parts of the function. Secondly,
we do not need switch back and forth between the Palm probabilities
with respect to the arrival process and with respect to the
regeneration points as in \cite{KonZaz:1}. However, we retain the
fundamental idea of \cite{KonZaz:1} by starting with its expression in
terms of the Palm probability $\PP^0$.

\begin{proof}{of Theorem~\ref{thm:princ}}
Assume that $f(0)=0$, so that $f$ is non-negative. The Palm inversion
formula~(\ref{eq:inversion}) gives
\begin{eqnarray*}
\EE f(W_\theta(0))
&=& \lambda\EE^0\intsurt f(W_\theta(t))\,dt\\
&=& \lambda\EE^0\intsurt\int_{\RR_+}\1_{\{W_\theta(t)>x\}}
  \,\mu_f(dx)\,dt\\
&=& \lambda\EE^0\int_{\RR_+}\intsurt\1_{\{W_\theta(t)>x\}}
  \,dt\,\mu_f(dx)
\end{eqnarray*}
and therefore
\begin{eqnarray*}
\longeqn {1\over h}\EE[f(W_{\theta+h}(0))-f(W_\theta(0))]\\
 &=& {\lambda\over h}\EE^0\int_{\RR_+}\intsurt
  [\1_{\{W_{\theta+h}(t)>x\}}-\1_{\{W_\theta(t)>x\}}]\,dt\,\mu_f(dx).
\end{eqnarray*}

In order to simplify the notations, let:
\begin{eqnarray*}
\varphi(x,t)
  &\egaldef&\1_{\{W_{\theta+h}(t)>x\}}-\1_{\{W_\theta(t)>x\}}\\
\Phi(\theta,h)
  &\egaldef& \int_{\RR_+}\intsurt\varphi(x,t)\,dt\,\mu_f(dx).
\end{eqnarray*}

The first step of our proof is to compute
$\lim_{h\to0}\Phi(\theta,h)/h$. We will have to integrate a function
taking its values in $\{-1,0,1\}$ with respect to $dt\,\mu_f(dx)$.
Define also for any $t\in[0,T_1)$:
$$
\DW \egaldef W_{\theta+h}(t)-W_\theta(t).
$$

\begin{figure}
\centering
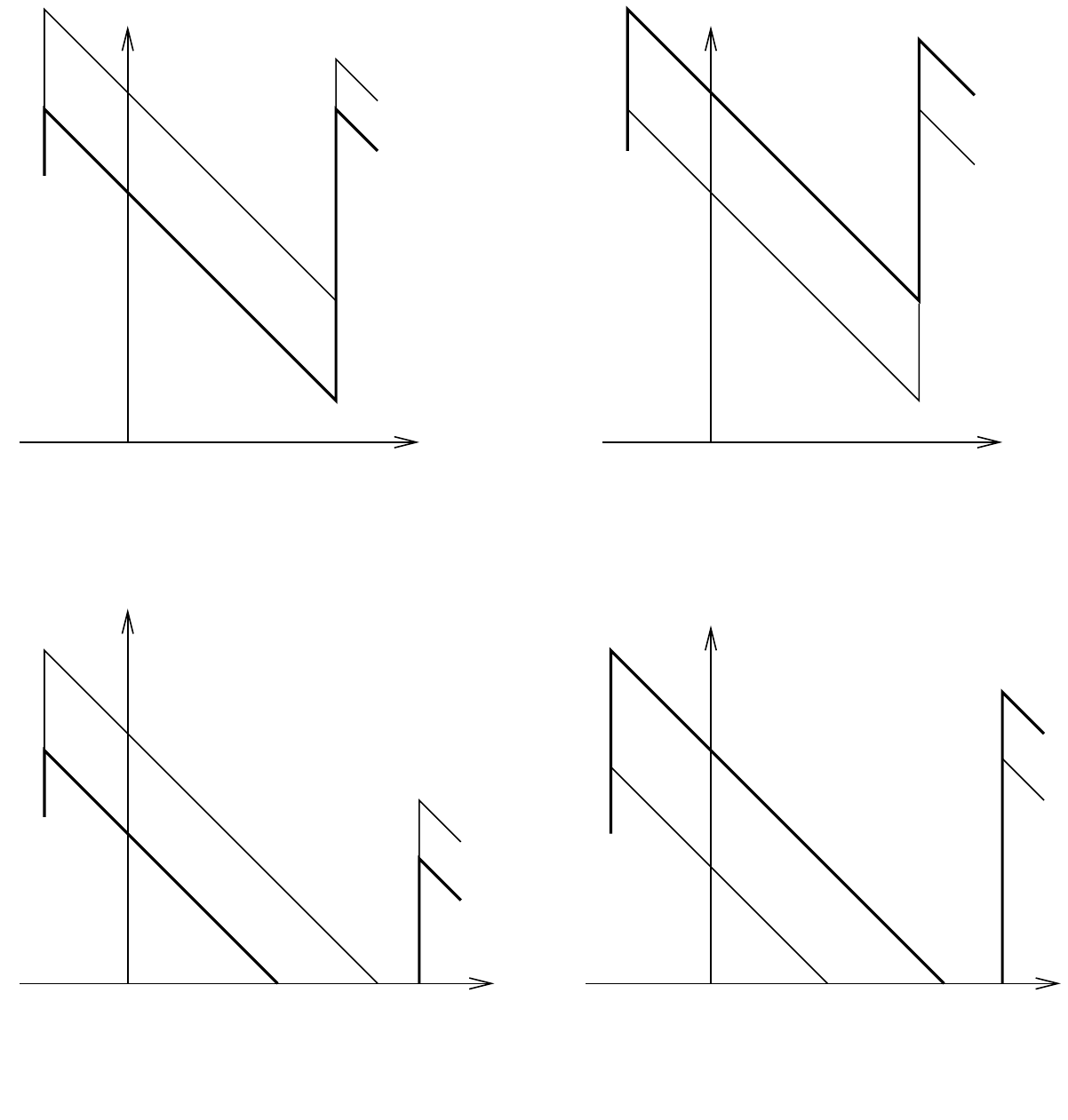
\caption{Four different cases for the computation of $\Phi$.}
\label{fig-cs}
\end{figure}
Assume first that $h>0$. As shown in Figure \ref{fig-cs}, we must
consider different cases depending on the relative position of
$W_\theta(0)$ and $W_{\theta+h}(0)$. We have to add cases 3 and 3$'$,
where $W'_\theta(0)=0$, preventing us to guess their relative
positions. In fact, all the terms of the formula can be found in the
first two cases and we will leave the other ones to the reader's
attention.

\begin{envthm}{Case 1:}
\begin{figure}[bth]
\centering
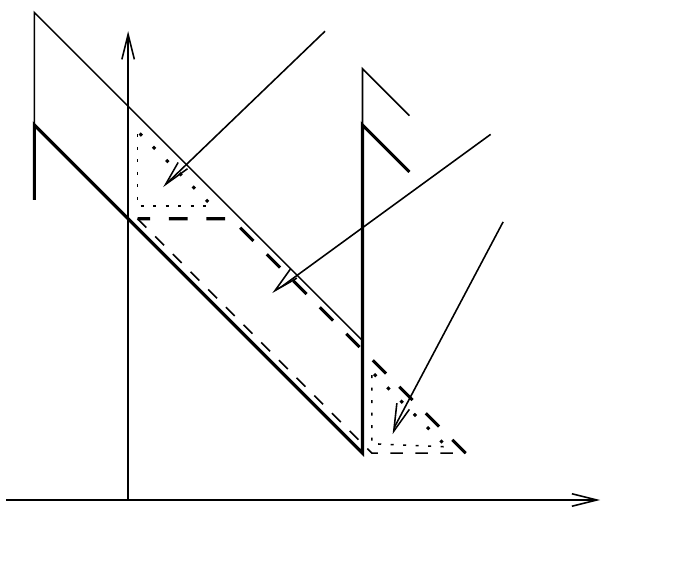
\caption{computation of $\Phi$ in case 1.}
\label{fig-c1}
\end{figure}
for $h$ small enough, $W_{\theta+h}(0)>W_\theta(0)$ and $\varphi=1$.
The way to compute $\Phi(\theta,h)$ can be best understood with the
help of Figure~\ref{fig-c1}. $\Phi$ is equal to the area with a dashed
border plus the dotted triangle on the left, minus the right one. Here
the borders included in the areas are in bold; since all functions are
{\em c\`adl\`ag}, these borders are the top and right ones.
\begin{eqnarray}\label{eq:tp:case1}
{1\over h}\Phi(\theta,h)
&=& \Bigl[f(W_\theta(0))-f(W_\theta(T_1-))\Bigr]
  {\DW \over h}\nonumber\\[2mm]
& & {}+{1\over h}\intDW
    \mu_f((W_\theta(0),W_\theta(0)+y])\,dy\nonumber\\[2mm]
& & {}-{1\over h}\intDW
    \mu_f((W_\theta(T_1-),W_\theta(T_1-)+y])\,dy.
\end{eqnarray}

The first term converges to
$[f(W_\theta(0))-f(W_\theta(T_1-))]W'_\theta(0)$. Moreover,
$$
\mu_f((W_\theta(0),W_\theta(0)+y])=\mu_f((W_\theta(0),W_\theta(0)+y))+\mu_f(\{W_\theta(0)+y\})
$$
and since $\mu_f(\{W_\theta(0)+y\})=0$ $dy$-a.e., the second term of
r.h.s.\ of equation (\ref{eq:tp:case1}) reads
$$
{1\over h}\intDW
    \mu_f((W_\theta(0),W_\theta(0)+y))\,dy,
$$
which is less or equal than
$$
(W'_\theta(0)+o(1))\,\mu_f((W_\theta(0),W_{\theta+h}(0))).
$$

Since $W_\theta(0)$ is continuous in the neighborhood of $\theta$,
this goes to zero with $h$. The third term converges to 0 for
the same reasons. So we have in this case:
$$
\lim_{h\to0+}{1 \over h}\Phi(\theta,h)
=[f(W_\theta(0))-f(W_\theta(T_1-))]W'_\theta(0).
$$
\end{envthm}

\begin{envthm}{Case 2:}
\begin{figure}
\centering
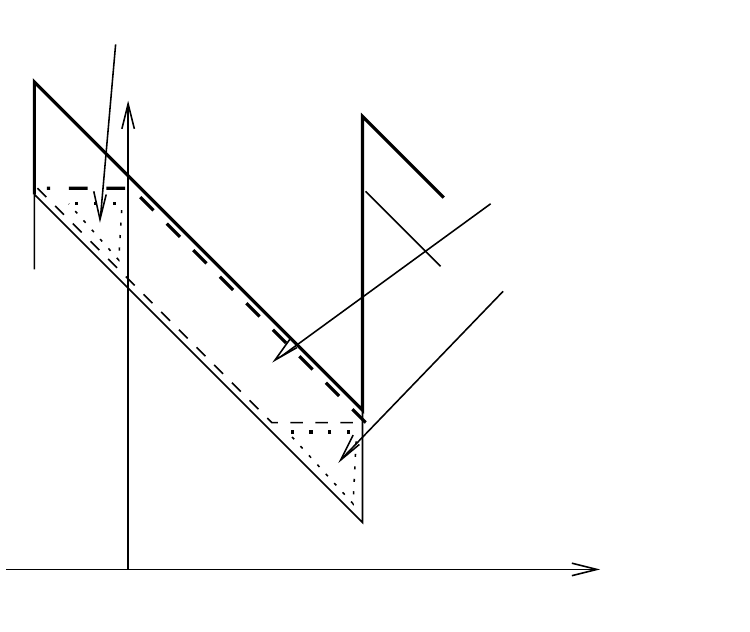
\caption{computation of $\Phi$ in case 2.}
\label{fig-c2}
\end{figure}
here $W_{\theta+h}(0)<W_\theta(0)$ and $\varphi=-1$. Due to the order
of $W_{\theta+h}(0)$ and $W_\theta(0)$, we find a formula different
from equation (\ref{eq:tp:case1})---see Figure~\ref{fig-c2}:
\begin{eqnarray}\label{eq:tp:case2}
{1\over h}\Phi(\theta,h)
&=&
-\biggl\{\Bigl[f(W_\theta(0))-f(W_\theta(T_1-))\Bigr]{-\DW \over h}\nonumber\\[2mm]
& & \mbox{}-{1\over h}\intmDW
    \mu_f((W_\theta(0)-y,W_\theta(0)])\,dy\nonumber\\[2mm]
& & \mbox{}+{1\over h}\intmDW
    \mu_f((W_\theta(T_1-)-y,W_\theta(T_1-)])\,dy\biggl\}.
\end{eqnarray}

The first term is the same as in case 1, but the second is equal to
$$
{1\over h}\intmDW\mu_f((W_\theta(0)-y,W_\theta(0)))\,dy
  \;-\;\mu_f(\{W_\theta(0)\}){\DW\over h}
$$
and its limit is $\mu_f(\{W_\theta(0)\})W'_\theta(0)$. The last term
of $\Phi(\theta,h)/h$ is computed in a similar way. Finally:
\begin{eqnarray*}
 \lim_{h\to0+}{1 \over h}\Phi(\theta,h)
& = & \Bigl[f(W_\theta(0))-f(W_\theta(T_1-))\\
&   & \mbox{\quad}-\mu_f(\{W_\theta(0)\})
                  +\mu_f(\{W_\theta(T_1-)\})\Bigr]W'_\theta(0).
\end{eqnarray*}
\end{envthm}

We can summarize the above cases in the following formula:
\begin{eqnarray*}
 \lim_{h\to0+}{1 \over h}\Phi(\theta,h)
& = & W'_\theta(0)\Bigl[f(W_\theta(0))-f(W_\theta(T_1-))\\
&   & \mbox{}-\1_{\{W'_\theta(0)<0\}}[\mu_f(\{W_\theta(0)\})
         -\mu_f(\{W_\theta(T_1-)\})]\Bigr].
\end{eqnarray*}

The next step is to find a bound for $\Phi(\theta,h)/h$ which has a
finite mean with respect to $\PP^0$. The formulas for each case give:
\begin{eqnarray*}
\Bigl|{1 \over h}\Phi(\theta,h)\Bigr|
&\leq& \Bigl(f(W_\theta(0))-f(W_\theta(T_1-))\Bigr)
       \Bigl|{\DW\over h}\Bigr|\\
& &\mbox{}+|f(W_{\theta+h}(0))-f(W_\theta(0))|
   \cdot\Bigl|{\DW\over h}\Bigr|\\
& &\mbox{}+|f(W_{\theta+h}(T_1-))-f(W_\theta(T_1-))|
   \cdot\Bigl|{\DW\over h}\Bigr|\\
&\leq& 3f(W^*(0))K^W_\theta(0).
\end{eqnarray*}
The last inequality takes advantage of the fact that $f$ is
non-decreasing and of the domination property (\ref{eq:domi}).
$K^W_\theta(t)$ is a Lipschitz coefficient for $W(t)$ w.r.t.\
$\theta$. Finally,
$$
\Bigl|{1 \over h}\Phi(\theta,h)\Bigr|
\leq 3f(W^*(0))K^W_\theta(0).
$$

The latter expression is independent from $h$. Moreover, it has a
finite mean under $\PP^0$: from Cauchy-Schwartz inequality,
$$
\EE^0\Bigl[f(W^*(0))K^W_\theta(0)\Bigr]
\leq \sqrt{\EE^0[f(W^*(0))]^2}\sqrt{\EE^0[K^W_\theta(0)]^2}.
$$

The first mean is finite from assumption \ref{ass:tp}-\romiii. To
prove that the second one is also finite, we must first give an
expression of $K^W_\theta(0)$:

\begin{eqnarray*}
\Bigl|{W_{\theta+h}(0)-W_\theta(0)\over h}\Bigr|
& \leq & \Bigl|{W_{\theta+h}(T_{-1})-W_\theta(T_{-1})\over h}\Bigr|
         +\Bigl|{\sigma_0(\theta+h)-\sigma_0(\theta)\over h}\Bigr|\\
& \leq & \sum_{n\in\ZZ}
          \Bigl|{\sigma_n(\theta+h)-\sigma_n(\theta)\over h}\Bigr|
          \1_{[R_-^*(T_0),0)}(T_n)\\
& \leq & \sum_{n\in\ZZ}
         K^\sigma(\xi_n)\1_{[R^*_-(0),R^*_+(0))}(T_n)\\
& \egaldef & K^W_\theta(0).
\end{eqnarray*}

The first inequality comes from equation (\ref{eq:lindley}) and
inequality $|a^+-b^+|\leq|a-b|$; then we use the boundary property
(\ref{eq:bound}) and last the Lipschitz property
\ref{ass:service}-\romi. To prove that $\EE^0[K^W_\theta(0)]^2$ is
finite, we can use the inequality $(x_1+\cdots+x_n)^p\leq
n^{p-1}(x_1^p+\cdots+x_n^p)$ and
\begin{eqnarray*}
\longeqn \EE^0\Bigl[\sum_{n\in\ZZ}
 K^\sigma(\xi_n)\1_{[R^*_-(0),R^*_+(0))}(T_n)\Bigr]^2\\
&\leq& \EE^0\Bigl[\sum_{n\in\ZZ}A([R^*_0,R^*_1))
[K^\sigma(\xi_n)]^2 \1_{[R^*_-(0),R^*_+(0))}(T_n)\Bigr]\\
&\leq& \EE^0[A([R^*_0,R^*_1))]^2[K^\sigma(\xi_0)]^2\\[2mm]
&\leq& \sqrt{\EE^0[A([R^*_0,R^*_1))]^4}
       \sqrt{\EE^0[K^\sigma(\xi_0)]^4},
\end{eqnarray*}
which is finite from \ref{ass:tp}-\romi\ and \ref{ass:tp}-\romii.
Here, the second inequality uses Lemma~\ref{lem:sum}.

Summing up our results, we can apply the Dominated Convergence
Theorem:
\begin{eqnarray*}
J'_r(\theta)
&\egaldef& \lim_{h\to0+}\EE[f(W_{\theta+h}(0))-f(W_\theta(0))]\\
&=&        \lim_{h\to0+}\lambda\EE^0{1\over h}\Phi(\theta,h)\\
&=&        \lambda\EE^0\lim_{h\to0+}{1\over h}\Phi(\theta,h)
\end{eqnarray*}

This gives equation (\ref{eq:tp:Jr}). The case of $h<0$ is handled in
the same way and gives equation (\ref{eq:tp:Jl})---loosely speaking,
the above cases used the sign of $W_{\theta+h}(0)-W_\theta(0)$; this
sign is inverted if $h<0$. This concludes the proof of the theorem.
\end{proof}

\begin{remark}
Assumption \ref{ass:tp} ensures that
$\EE^0\bigl[f(W^*(0))K^W_\theta(0)\bigr]<\infty$. If we know that $f$
is bounded, for example, the only assumptions we need are
\begin{enumerate}
\item $\EE^0[K^\sigma(\xi_0)]^2<\infty;$
\item $\EE^0[A([R_0^*,R_1^*))]^2<\infty.$
\end{enumerate}
This reduced set of assumptions can for instance be used in
Example~\ref{ex:proba}.
\end{remark}

It is important to point out that Corollary~\ref{cor:princ} cannot
always be applied. We show such a case in next example :

\begin{figure}
\centering
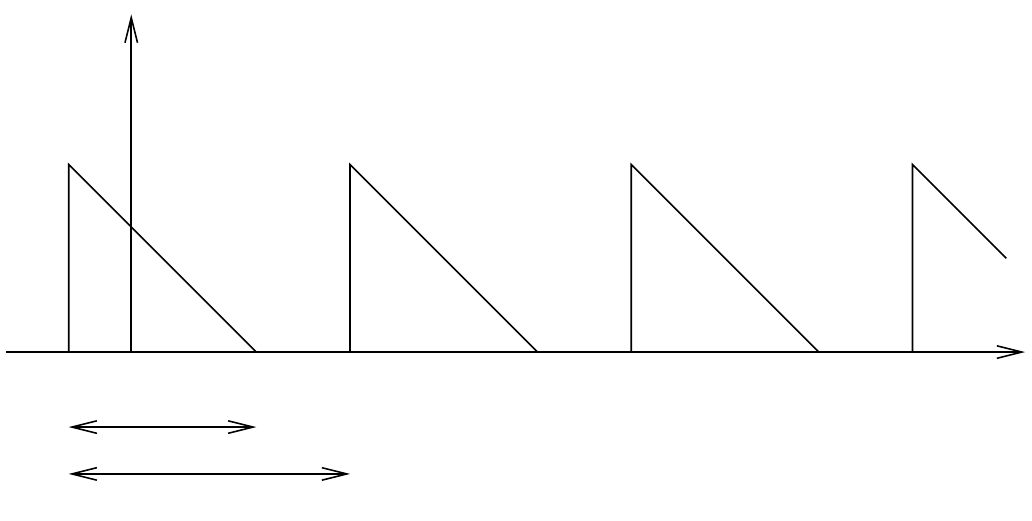
\caption{workload of the D/D/1 queue.}
\label{fig-deter}
\end{figure}

\begin{example}
Consider a D/D/1 queue, that is with deterministic inter-arrival time
$\tau$ and service time $\theta<\tau$. In order to have a stationary
queue, $T_0$ must be uniformly spread in $[-\tau,0]$. As we can see in
Figure~\ref{fig-deter}, we have
$$
W(T_n)=\theta,\ W(T_n-)=0,\ W'_\theta(0)=1\ \PP^0\mbox{-a.s.}
$$

For $x>0$, take $f(w)=\1_{\{w\geq x\}}$ as in Example~\ref{ex:proba}.
Then if $\theta\leq x$, $J(\theta)=0$, else
$$
J(\theta)=\int_{-\tau}^0\1_{\{\theta+t\geq x\}}{dt \over \tau}
  ={\theta-x\over\tau}.
$$

Finally $J(\theta)\equiv \PP(W_\theta(0)\geq x)
=\left({\theta-x\over\tau}\right)^+$, which is not differentiable at
point $\theta=x$. Besides,
\begin{eqnarray*}
J'_r(\theta)&=& \lambda\EE^0[\1_{\{\theta\geq x\}}-\1_{\{0\geq x\}}]\\
            &=& {1\over \tau}\1_{\{\theta\geq x\}}\\
J'_l(\theta)&=& \lambda\EE^0[\1_{\{\theta\geq x\}}-\1_{\{0\geq x\}}
                 +\1_{\{0=x\}}-\1_{\{\theta=x\}}]\\
            &=& {1\over \tau}\1_{\{\theta> x\}}
\end{eqnarray*}
\end{example}

\section{Second order derivative.}\label{sec:2ndorder}

The method used in Section~\ref{sec:service} can be used for
higher-order derivatives. We need assumptions on the properties of our
system and some new moment conditions:

\begin{ass}\label{ass:der2:struct}
$G$ and $f$ verify the following:
\begin{enumerate}
\item $\theta\mapsto G(\xi,\theta)$ is twice differentiable and
there exists a function $\xi\mapsto K^{\sigma'}(\xi)$ such that
$$
|G(\xi,\theta+2h)-2G(\xi,\theta+h)+G(\xi,\theta)| \leq
h^2K^{\sigma'}(\xi); 
$$
\item $w\mapsto f(w)$ is non-decreasing and differentiable.
\end{enumerate}
\end{ass}

\begin{ass}\label{ass:der2:moments}
The following inequalities hold:
\begin{enumerate}
\item $\EE^0[K^\sigma(\xi_0)]^8<\infty;$
\item $\EE^0[K^{\sigma'}(\xi_0)]^4<\infty;$
\item $\EE^0[A([R^*_0,R^*_1))]^8<\infty;$
\item $\EE^0[f(W^*(0))]^2<\infty;$
\item $\EE^0[\sup_\theta f'(W_\theta(0))]^2<\infty;$
\item $\EE^0[\sup_\theta f'(W_\theta(T_1-))]^2<\infty.$
\end{enumerate}
\end{ass}

The main result of this section is:

\begin{thm}\label{thm:der2}
Assume \ref{ass:service}, \ref{ass:der2:struct} and
\ref{ass:der2:moments} hold; then $J$ admits a right second derivative
with respect to $\theta$ given by
\begin{eqnarray}
J''_r(\theta)
& = & \lambda\EE^0\biggl[W''_\theta(0)
              \Bigl[f(W_\theta(0))-f(W_\theta(T_1-))\Bigr]\nonumber\\
&   & \mbox{\qquad}+[W'(0)]^2
              \Bigl[f'(W_\theta(0))-f'(W_\theta(T_1-))\nonumber\\
&   & \mbox{\qquad}-\1_{\{W'_\theta(0)<0\}}
         [\mu_{f'}(\{W_\theta(0)\})-\mu_{f'}(\{W_\theta(T_1-)\})]
    \Bigr]\biggr],\label{eq:der2:Jr}
\end{eqnarray}
and its left second derivative is
\begin{eqnarray}
J''_l(\theta)
& = & \lambda\EE^0\biggl[W''_\theta(0)
              \Bigl[f(W_\theta(0))-f(W_\theta(T_1-))\Bigr]\nonumber\\
&   & \mbox{\qquad}+[W'(0)]^2
              \Bigl[f'(W_\theta(0))-f'(W_\theta(T_1-))\nonumber\\
&   & \mbox{\qquad}-\1_{\{W'_\theta(0)>0\}}
         [\mu_{f'}(\{W_\theta(0)\})-\mu_{f'}(\{W_\theta(T_1-)\})]
    \Bigr]\biggr].\label{eq:der2:Jl}
\end{eqnarray}
\end{thm}

\begin{cor}
Assume \ref{ass:service}, \ref{ass:der2:struct} and
\ref{ass:der2:moments} hold; if $f'$ is continuous or if $W_\theta(0)$
and $W_\theta(T_1-)$ admit densities w.r.t.\ $\PP^0$ then $J(\theta)$
is differentiable twice and
\begin{eqnarray*}
J''(\theta)
& = & \lambda\EE^0\biggl[W''_\theta(0)
              \Bigl[f(W_\theta(0))-f(W_\theta(T_1-))\Bigr]\\
&   & \mbox{\qquad}+[W'(0)]^2
              \Bigl[f'(W_\theta(0))-f'(W_\theta(T_1-))
    \Bigr]\biggr].
\end{eqnarray*}
\end{cor}

\begin{proof}{of Theorem~\ref{thm:der2}}
As this proof is very similar to that of Theorem~\ref{thm:princ}, we
will omit the parts of it which are not new. We want to compute the
limit as $h\to 0$ of 
$$
{1\over h^2}\EE
  \Bigl[f(W_{\theta+2h}(0))-2f(W_{\theta+h}(0))+f(W_\theta(0))\Bigr]
 ={\lambda\over h^2}\EE^0\Phi_2(\theta,h),
$$
with
\begin{eqnarray*}
\Phi_2(\theta,h)& \egaldef& 
\int_{\RR_+}\intsurt\Bigl[
   [\1_{\{W_{\theta+2h}(t)>x\}}-\1_{\{W_{\theta+h}(t)>x\}}]\\
&  & \mbox{\qquad}-[\1_{\{W_{\theta+h}(t)>x\}}-\1_{\{W_\theta(t)>x\}}]
        \Bigr]dt\,\mu_f(dx).
\end{eqnarray*}

We will once more distinguish two important cases among all possible
ones, depending on the sign of $W'_\theta(0)$. Suppose first that
$h>0$.

\begin{envthm}{Case 1:}
\begin{figure}
\centering
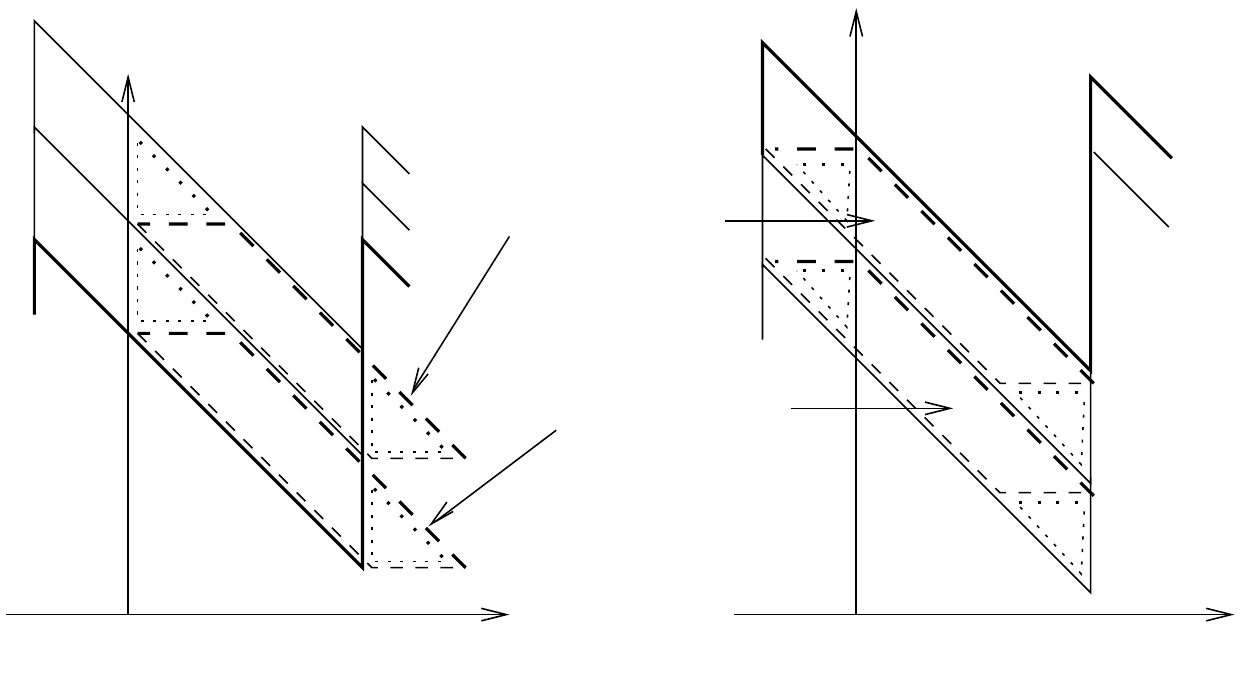
\caption{computation of $\Phi_2$ in cases 1 and 2}\label{fig-d2}
\end{figure}
$W'_\theta(0)>0$; for $h$ small enough, 
$W_\theta(0)<W_{\theta+h}(0)<W_{\theta+2h}(0)$---see Figure
\ref{fig-d2}. We have here to subtract the areas of two bands
which are of the same sort as in Theorem~\ref{thm:princ}:
\begin{eqnarray*}
\Phi_2(\theta,h)
& = & \DWh\Bigl[f(W_{\theta+h}(0))-f(W_{\theta+h}(T_1-))\Bigr]\\
&   & \mbox{}-\DW\Bigl[f(W_\theta(0))-f(W_\theta(T_1-))\Bigr]\\[2mm]
&   & \mbox{}+A_{\theta,h}(0)-A_{\theta,h}(T_1-)
\end{eqnarray*}
where
\begin{eqnarray*}
A_{\theta,h}(t)
& = & \intDWh\mu_f((W_{\theta+h}(t),W_{\theta+h}(t)+y])\,dy\\
&   & \mbox{}-\intDW\mu_f((W_\theta(t),W_\theta(t)+y])\,dy.
\end{eqnarray*}

The main term is equal to
\begin{eqnarray*}
&&\DWh\Bigl[f(W_{\theta+h}(0))-f(W_{\theta+h}(T_1-))
 -f(W_{\theta}(0))+f(W_{\theta}(T_1-))\Bigr]\\
&&\mbox{}+\DDW\Bigl[f(W_{\theta}(0))-f(W_{\theta}(T_1-))\Bigr].
\end{eqnarray*}

Moreover,
\begin{eqnarray*}
A_{\theta,h}(0)
& = & \intDW\unskip\Bigl[f(W_{\theta+h}(0)+y)-f(W_{\theta+h}(0))\\
&   & \mbox{\qquad}-f(W_{\theta}(0)+y)+f(W_{\theta}(0))\Bigr]dy 
      + o(h^2)\\
& = & \intDW hW'_\theta(0)\mu_{f'}((W_\theta(0),W_\theta(0)+y])dy+o(h^2)
\end{eqnarray*}

As in Theorem~\ref{thm:princ}-case 1, 
$\lim_{h\to0}A_{\theta,h}(0)/h^2=0$; the limit is the same for
$A_{\theta,h}(T_1-)$. Consequently,
\begin{eqnarray*}
\lim_{h\to0+}{1\over h^2}\Phi_2(\theta,h)
& = & [W'_\theta(0)]^2
         \Bigl[f'(W_{\theta}(0))-f'(W_{\theta}(T_1-))\Bigr]\\
&   & \mbox+W''_\theta(0)
         \Bigl[f(W_{\theta}(0))-f(W_{\theta}(T_1-))\Bigr].
\end{eqnarray*}

\end{envthm}

\begin{envthm}{Case 2:}
$W'_\theta(0)<0$; for $h$ small enough, 
$W_\theta(0)>W_{\theta+h}(0)>W_{\theta+2h}(0)$ and
\begin{eqnarray*}
\Phi_2(\theta,h)
& = & -1\cdot\biggl\{-\DWh
      \Bigl[f(W_{\theta+h}(0))-f(W_{\theta+h}(T_1-))\Bigr]\\
&   & \mbox{\qquad}-\DW
      \Bigl[f(W_\theta(0))-f(W_\theta(T_1-))\Bigr]\\[2mm]
&   & \mbox{\qquad}-B_{\theta,h}(0)+B_{\theta,h}(T_1-)\biggr\}
\end{eqnarray*}
where
\begin{eqnarray*}
B_{\theta,h}(t)
& \egaldef & \intmDWh\mu_f((W_{\theta+h}(t)-y,W_{\theta+h}(t)])\,dy\\
&          & \mbox{}-\intmDW\mu_f((W_\theta(t)-y,W_\theta(t)])\,dy.
\end{eqnarray*}

While the main part has the same limit as in case 1, we have
\begin{eqnarray*}
B_{\theta,h}(0)
& = & \intmDW hW'_\theta(0)\mu_{f'}((W_\theta(0)-y,W_\theta(0)])dy+o(h^2)\\
& = & -h^2[W'_\theta(0)]^2\mu_{f'}(\{W_\theta(0)\})+o(h^2).
\end{eqnarray*}

Finally, in case 2,

\begin{eqnarray*}
\lim_{h\to0+}{1\over h^2}\Phi_2(\theta,h)
& = & [W'_\theta(0)]^2
         \Bigl[f'(W_{\theta}(0))-f'(W_{\theta}(T_1-))\Bigr]\\
&   & \mbox{}+W''_\theta(0)
         \Bigl[f(W_{\theta}(0))-f(W_{\theta}(T_1-))\Bigr]\\
&   & \mbox{}-[W'_\theta(0)]^2
         \Bigl[\mu_{f'}(\{W_\theta(0)\})
              -\mu_{f'}(\{W_\theta(T_1-)\})\Bigr].
\end{eqnarray*}

\end{envthm}

Besides,
\begin{eqnarray*}
\Bigl|{1\over h^2}\Phi_2(\theta,h)\Bigr|
& \leq & 3\Bigl[\sup_\theta f'(W_\theta(0))
               +\sup_\theta f'(W_\theta(T_1-))\Bigr]
            \Bigl[K^W_\theta(0)\Bigr]^2\\
&      & \mbox{}+f(W^*(0))K^{W'}_\theta(0),
\end{eqnarray*}
where $K^W_\theta(0)$ is the same as in Theorem~\ref{thm:princ} and
$$
K^{W'}_\theta(0) \egaldef \sum_{n\in\ZZ}
            K^{\sigma'}(\xi_n)\1_{[R^*_-(0),R^*_+(0))}(T_n).
$$

As in theorem \ref{thm:princ}, we use Lemma \ref{lem:sum},
Cauchy-Schwartz inequality and assumption \ref{ass:der2:moments} to
prove that $|\Phi_2(\theta,h)/h^2|$ has a finite mean under $\PP^0$.
Using the Dominated Convergence Theorem, we find expressions
(\ref{eq:der2:Jr}) and (\ref{eq:der2:Jl}) for the second derivatives
of $J$.
\end{proof}

\section{Other parameters of the queue.}\label{sec:other}
\def\DW{{\Delta W_{\nu,h}}}
\def\bW{\bar{W}}

Let us consider a setting slightly different from the original one: we
still deal with a G/G/1 queue, but now working at speed $\nu$.
Lindley's equation for the workload of the queue reads:
$$
W_\nu(t)=\Bigl(W_\nu(T_n-)+\sigma_n-\nu(t-T_n)\Bigr)^+,\
t\in[T_n,T_{n+1}).
$$

We address the same problem as in Section~\ref{sec:service} in this
new setting. Our method can apply in this case in the same way as for
variable service times; we will try to keep the notations as close as
possible to those of Section~\ref{sec:construction} to point out the
similitudes, replacing $\theta$ with $\nu$ when necessary.

\begin{remark}
As an anonymous reviewer pointed out, if we define $\bW_\nu(t)=\nu
W_\nu(t)$ we have the relation: 
$$
\bW_\nu(t)=\Bigl(\bW_\nu(T_n-)+{\sigma_n\over\nu}-(t-T_n)\Bigr)^+,\
t\in[T_n,T_{n+1}).
$$
This means that the  queue with workload $\bW_\nu$ fits in the
framework of Sections~\ref{sec:construction} and~\ref{sec:service}.
Nevertheless, what we want to estimate is $(\partial/\partial\nu)\EE f(\nu
\bW_\nu(0))$, which does not follow directly from
Theorem~\ref{thm:princ}. The extra computations needed would cancel
the gain of using Theorem~\ref{thm:princ}. Note also that this result
will be useful in the second part of this section to deal with
parameters of the arrival process.
\end{remark}

Assume that $\nu\geq \nu^*>0$; then as in
Section~\ref{sec:construction}, we can construct all the queues for
different values of $\nu$ so that for all $\nu\geq\nu^*$ and
$t\in\RR$, we have the relation (see Figure~\ref{fig-spee-domi})
\begin{figure}
\centering
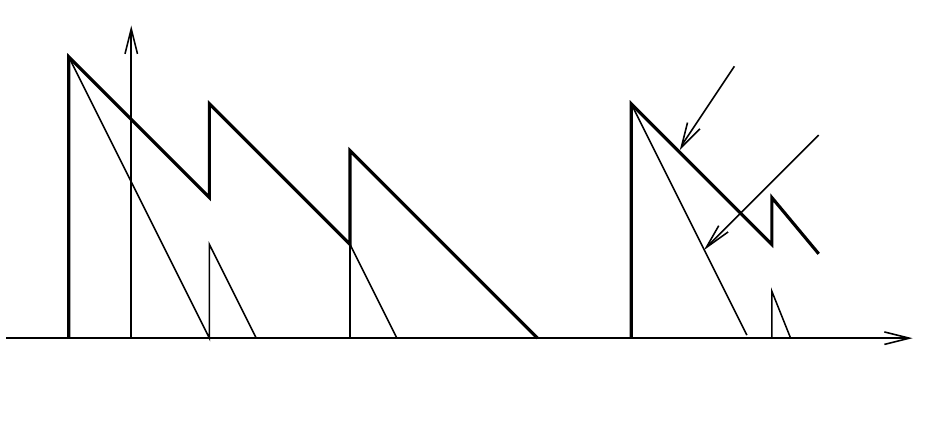
\caption{the domination property for the speed}
\label{fig-spee-domi}
\end{figure}
\begin{equation}\label{eq:spee:domi}
W_\nu(t)\leq W_{\nu^*}(t)
\end{equation}
\begin{equation}\label{eq:spee:bound}
R^*_-(t)\leq R_-(\nu)(t)\leq t < R_+(\nu)(t)\leq R^*_+(t),
\end{equation}

The assumption on moments we need is much like \ref{ass:tp}:

\begin{ass}\label{ass:spee}
The following moments are finite:
\begin{enumerate}
\item $\EE^0[\tau_0]^4<\infty;$
\item $\EE^0[A([R^*_0,R^*_1))]^4<\infty;$
\item $\EE^0[f(W_{\nu^*}(0))]^2<\infty.$
\end{enumerate}
\end{ass}

The first real difference with the results of
Section~\ref{sec:service} is that the expressions for the derivative
use a primitive of $f$, whereas only $f$ appeared in
Theorem~\ref{thm:princ}.

\begin{thm}\label{thm:spee}
Let $F$ be a primitive of $f$; if \ref{ass:spee} holds, then $J$ has a
right-hand derivative equal to:
\begin{eqnarray}\label{eq:thm:spee}
J'_r(\nu)
& = & {\lambda\over\nu^2}\EE^0\biggl\{
      \nu W'_\nu(0)\Bigl[f(W_\nu(0))-f(W_\nu(T_1-))\Bigr]\nonumber\\
&   & \mbox{}+F(W_\nu(0))-F(W_\nu(T_1-))\nonumber\\
&   & \mbox{}-[W_\nu(0)-W_\nu(T_1-)]f(W_\nu(T_1-))\biggr\}
\end{eqnarray}
and its left-hand derivative is 
\begin{eqnarray*}
J'_l(\nu)
& = & {\lambda\over\nu^2}\EE^0\biggl\{
      \nu W'_\nu(0)\Bigl[f(W_\nu(0))-f(W_\nu(T_1-))\Bigr]\\
&   & \mbox{}+F(W_\nu(0))-F(W_\nu(T_1-))\\[1mm]
&   & \mbox{}-[W_\nu(0)-W_\nu(T_1-)]f(W_\nu(T_1-))\\[2mm]
&   &\mbox{}+ \nu W'_\nu(0)
           [\mu_f(\{W_\nu(0)\})-\mu_f(\{W_\nu(T_1-)\})]\\
&   &\mbox{}-[W_\nu(0)-W_\nu(T_1-)]\mu_f(\{W_\nu(T_1-)\})\biggr\}.
\end{eqnarray*}

If $f$ is continuous or if both $W_\nu(0)$ and $W_\nu(T_1-)$ admit
densities w.r.t.\ $\PP^0$, then $J$ is differentiable and its
derivative is equal to $J'_r$.
\end{thm}
\begin{remark}
The expressions in Theorem~\ref{thm:spee} seem really complicated when
compared to those obtained in Theorem~\ref{thm:princ}; in fact, in the
case where $f$ is differentiable, the inversion formula applied to
(\ref{eq:thm:spee}) gives the classic IPA formula
$$ J'(\nu)=\EE W'_\nu(0)f(W_\nu(0)).  $$ The complexity of
(\ref{eq:thm:spee}) comes from the fact that $W'_\nu(t)$ is not
constant on $[T_0,T_1)$.
\end{remark}

\begin{proof}{of Theorem \ref{thm:spee}}
We once more proceed as in the proof of Theorem~\ref{thm:princ}---more
details can be found in~\cite{BreLas:1}. Define
$$
\Phi(\nu,h)
  \egaldef \int_{\RR_+}\intsurt
     [\1_{\{W_{\nu+h}(t)>x\}}-\1_{\{W_\nu(t)>x\}}]
     \,dt\,\mu_f(dx)
$$
and remark that
$$
{1\over h}\EE[f(W_{\nu+h}(0))-f(W_\nu(0))]
={\lambda\over h}\EE^0\Phi(\nu,h).
$$

\begin{figure}
\centering
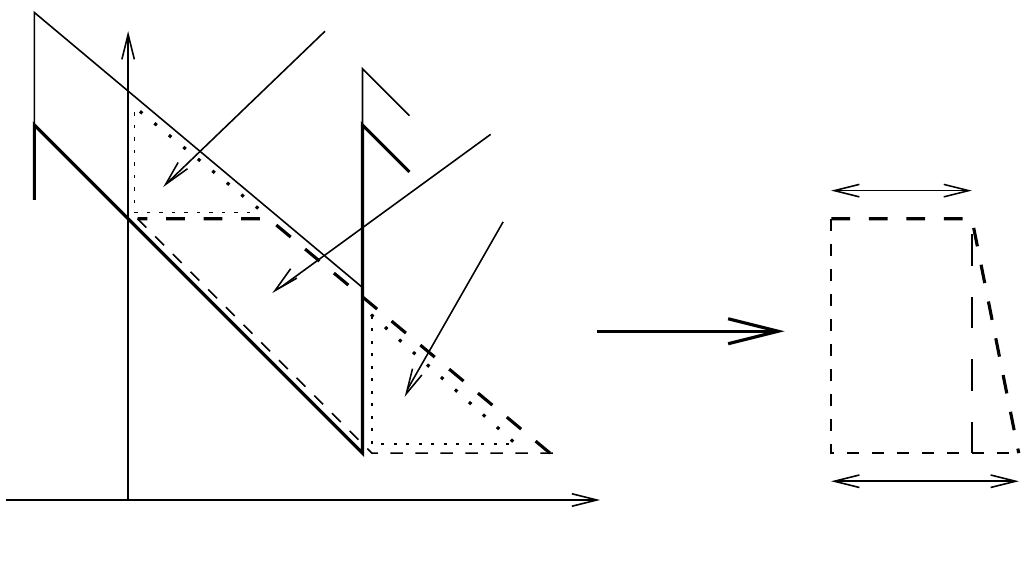
\caption{computation of $\Phi$ for $h>0$.}
\label{fig-spee-case1}
\end{figure}
We will consider only right-hand derivatives; left-hand derivatives
are obtained with the same method. Figure~\ref{fig-spee-case1} shows
how $\Phi$ can be computed: the main area is equal to the area of the
trapezium on the right. As $W_\nu$ is linear in $\nu$, we have
$$\DW(T_1-)-\DW(0) = h T'_1,$$ 
where 
$$
T'_1\egaldef
\min\Bigl[{W_\nu(0)\over\nu},T_1\Bigr]={W_\nu(0)-W_\nu(T_1-)\over\nu}.
$$

The area of the trapezium of Figure~\ref{fig-spee-case1} is equal to
\begin{eqnarray*}
{\cal A}
& = & {\DW(0)\over\nu+h}\Bigl[f(W_\nu(0))-f(W_\nu(T_1-))\Bigr]\\
&   & \mbox{}+\int_0^{hT'_1\over\nu+h}
        \mu_f\Bigl(\Bigl(W_\nu(T_1-),
            W_\nu(0)-{\nu(\nu+h)\over h}y\Bigr]\Bigr)\,dy\\
& = & {h\over\nu+h}\biggl\{{\DW(0)\over h}
        \Bigl[f(W_\nu(0))-f(W_\nu(T_1-))\Bigr]\\
&   & \mbox{\quad}+{1\over\nu}\Bigl[F(W_\nu(0))-F(W_\nu(T_1-))\Bigr]\\
&   & \mbox{\quad}-{W_\nu(0)-W_\nu(T_1-)\over\nu}
                   f(W_\nu(T_1-))\biggr\}.
\end{eqnarray*}

The additional terms read:
\begin{eqnarray*}
& & \int_0^{\DW(0)\over\nu+h}
          \mu_f((W_\nu(0),W_\nu(0)+(\nu+h)y])\,dy\\
& & -\int_0^{\DW(T_1-)\over\nu+h}
          \mu_f((W_\nu(T_1-),W_\nu(T_1-)+(\nu+h)y])\,dy.
\end{eqnarray*}

As we have shown in the proof of Theorem~\ref{thm:princ}, this kind of
expression is an $o(h)$ and
\begin{eqnarray*}
\lim_{h\to0+}\Phi(\nu,h)
& = & {W'_\nu(0)\over\nu}\Bigl[f(W_\nu(0))-f(W_\nu(T_1-))\Bigr]\\
&   & \mbox{}+{1\over\nu^2}
          \Bigl[F(W_\nu(0))-F(W_\nu(T_1-))\Bigl]\\[2mm]
&   & \mbox{}-{W_\nu(0)-W_\nu(T_1-)\over\nu}f(W_\nu(T_1-)).
\end{eqnarray*}

Moreover, as in Theorem~\ref{thm:princ}, we have
$$
\Bigl|{1\over h}\Phi(\nu,h)\Bigr|
 \leq {1\over\nu}[3K^W_\nu(0)+2\tau_0]f(W_{\nu^*}(0)),
$$
where $K^W_\nu(0)$ is a Lipschitz coefficient for $W_\nu(0)$ w.r.t.\
$\nu$, which can be expressed as in Theorem~\ref{thm:princ} as 
$$
K^W_\nu(0)\egaldef \sum_{n\in\ZZ}\tau_n\1_{[R^*_-(0),R^*_+(0))}(T_n).
$$

One can easily check that assumption \ref{ass:spee} suffices to prove
that $|\Phi/h|$ is bounded by an integrable variable. Consequently, we
can apply the Dominated Convergence Theorem and find the expected
result.
\end{proof}

\def\t#1{\tilde#1}
\def\wt#1{\widetilde#1}
\def\wtEE{\mathop{\wt{\EEbase}}\nolimits}
\def\wtPP{\mathop{\wt{\PPbase}}\nolimits}

The method used so far does not apply to the case where the parameter
of interest is a parameter of the inter-arrival times; in this case,
the Palm measure associated to the arrival process depends on the
parameter and the method fails. We show how a change of time scale can
be used in some cases. We consider a G/G/1 queue with inter-arrival
times $\procZ{\tau_n(\alpha)}$, $\alpha\geq\alpha^*>0$ and we will
restrict our attention to the following case:

\begin{ass}\label{ass:arrival}
$\alpha$ is a scale parameter for $\tau_n(\alpha)$, that is
$\tau_n(\alpha)=\alpha\eta_n.$
\end{ass}

Lindley's equation takes the form
\begin{equation}\label{eq:arri:lindley}
W_\alpha(t)
=\Bigl(W_\alpha(T_n(\alpha)-)+\sigma_n-(t-T_n(\alpha))\Bigr)^+,
    \ t\in[T_n(\alpha),T_{n+1}(\alpha))
\end{equation}

Now define a G/G/1 queue {\em with speed $\alpha$} which inter-arrival
times, service times and arrival process are given by:
\begin{eqnarray*}
\t\tau_n   & \egaldef & {\tau_n(\alpha)\over \alpha}\,=\,\eta_n\\
\t\sigma_n & \egaldef & \sigma_n\\
\wt{T}_n   & \egaldef & {T_n(\alpha)\over \alpha}.
\end{eqnarray*}

These processes are stationary with respect to the measurable flow
$\t\theta_t\egaldef \theta_{\alpha t}$ and the queue they define is
stable whenever the original one is; this queue will be referred to as
the ``auxiliary system''. Throughout this section, we will use the
same notations as for the main system, but with a tilde. Lindley's
equation for the auxiliary system reads:
\begin{equation}\label{eq:arri:lindleyaux}
\wt{W}_\alpha(t)
 =\Bigl(\wt{W}_\alpha(\wt{T}_n-)+\t\sigma_n-\alpha(t-\wt{T}_n)\Bigr)^+
  ,\ t\in[\wt{T}_n,\wt{T}_{n+1}).
\end{equation}

Comparing equations~(\ref{eq:arri:lindley})
and~(\ref{eq:arri:lindleyaux}) and noting that the process
$W_\alpha(\alpha t)$ is stationary with respect to the flow
$\t\theta_t$, uniqueness in Loynes' Stability Theorem---see Baccelli
and Br\'emaud~\cite{BacBre:1}--- yields
$$
W_\alpha(t)=\wt{W}_\alpha(t/\alpha).
$$

\begin{figure}
\centering
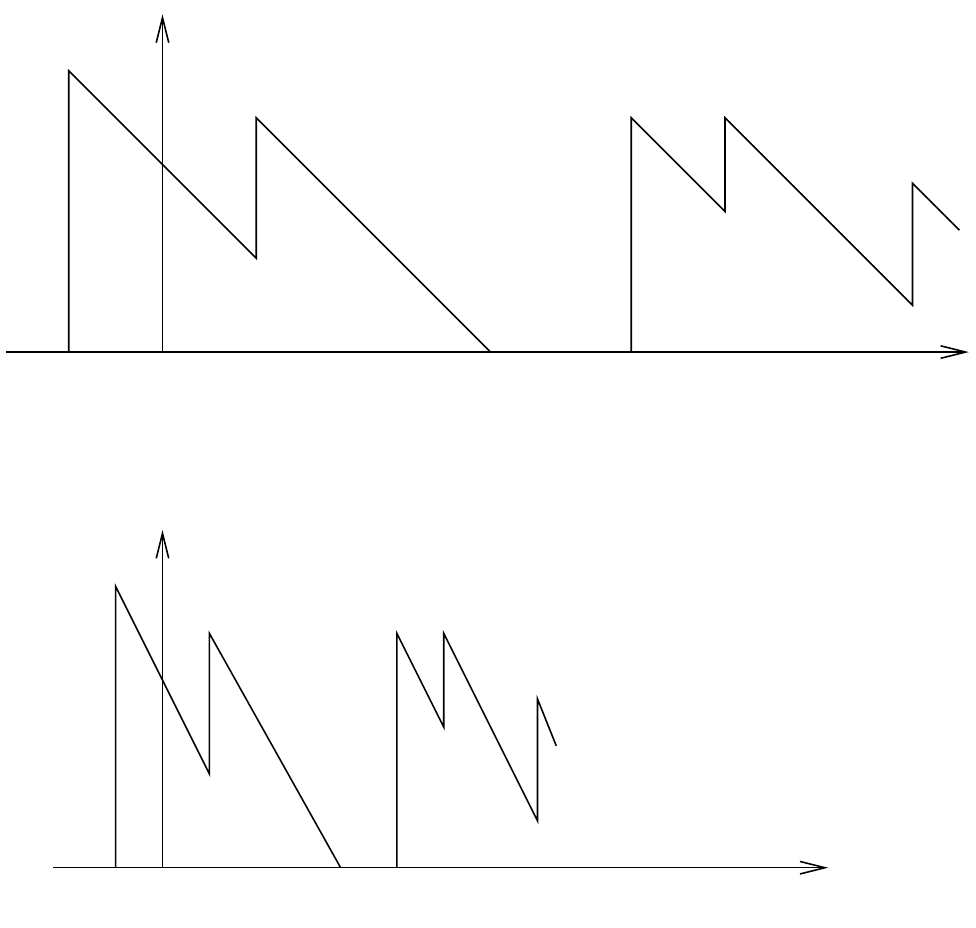
\caption{Change of time scale for $\alpha=2$.}
\label{fig-arri-scale}
\end{figure}
The effect of the change of time scale can be seen on
Figure~\ref{fig-arri-scale}. Moreover,
\begin{eqnarray*}
\t\lambda & = & \EE\wt{A}((0,1])\\
          & = & \EE A((0,\alpha])\,=\,\alpha\lambda(\alpha)\\
\wt{W}'_\alpha(0) & = & W'_\alpha(0).
\end{eqnarray*}

In the computation of $\t\lambda$, we use the fact that the auxiliary
system is defined on the same probability space than the main one. It
has its own Palm measure associated to $\procZ{\wt{T}_n}$, say
$\wtPP^0$. The way to switch between probability measures
$\PP^0_\alpha$ and $\wtPP^0$ will be shown in the proof of
Theorem~\ref{thm:arri}. Before proceeding, we need a set of
\ref{ass:spee}-like conditions:

\begin{ass}\label{ass:arri:moments}
The following conditions hold:
\begin{enumerate}
\item $\EE^0_\alpha[\tau_0]^4<\infty;$
\item $\EE^0_{\alpha^*}[A([R_0^*,R_1^*))]^4<\infty;$
\item $\EE^0_{\alpha^*}[f(W_{\alpha^*}(0))]^2
       <\infty.$
\end{enumerate}
\end{ass}

Using this model, we find the following result:

\begin{thm}\label{thm:arri}
Assume \ref{ass:arrival} and \ref{ass:arri:moments} hold; then 
\begin{eqnarray}
J'_r(\alpha)
& = & {\lambda\over\alpha}\EE^0_\alpha\biggl\{
      \alpha W'_\alpha(0)
      \Bigl[f(W_\alpha(0))-f(W_\alpha(T_1-))\Bigr]\nonumber\\
&   & \mbox{}+\Bigl[F(W_\alpha(0))-F(W_\alpha(T_1-))\Bigl]\nonumber\\
&   & \mbox{}-[W_\alpha(0)-W_\alpha(T_1-)]
              f(W_\alpha(T_1-))\biggr\}\label{eq:arri:Jr}\\
J'_l(\alpha)
& = & {\lambda\over\alpha}\EE^0_\alpha\biggl\{
      \alpha W'_\alpha(0)
      \Bigl[f(W_\alpha(0))-f(W_\alpha(T_1-))\Bigr]\nonumber\\
&   & \mbox{}+\Bigl[F(W_\alpha(0))-F(W_\alpha(T_1-))\Bigl]\nonumber\\
&   & \mbox{}-[W_\alpha(0)-W_\alpha(T_1-)]
              f(W_\alpha(T_1-))\nonumber\\
&   & \mbox{}-\alpha W'_\alpha(0)
      \Bigl[\mu_f(\{W_\alpha(0)\})-\mu_f(\{W_\alpha(T_1-)\})\Bigr]\nonumber\\
&   & \mbox{}+[W_\alpha(0)-W_\alpha(T_1-)]
              \mu_f(\{W_\alpha(T_1-)\})
      \biggr\}.\label{eq:arri:Jl}
\end{eqnarray}

If $f$ is continuous or if $W_\alpha(0)$ and $W_\alpha(T_1(\alpha)-)$
admit densities with respect to $\PP^0_\alpha$ then $J$ is
differentiable w.r.t.\ $\alpha$ and its derivative is equal to $J'_r$.
\end{thm}

\begin{proof}{}
We have
$$
J(\alpha)=\EE f(W_\alpha(0))=\EE f(\wt{W}_\alpha(0))
$$
where $\wt{W}_\alpha(0)$ is the workload of the auxiliary queue with
speed $\alpha$. We aim to apply Theorem~\ref{thm:spee} to this queue
and then adapt the result to the main queue. The three conditions of
\ref{ass:spee} correspond to the three ones of \ref{ass:arri:moments}: for condition \ref{ass:spee}-\romi, note that
\begin{eqnarray*}
\wtEE^0[\t\tau_0]^4
& = & {1\over \t\lambda}\EE\sum_{n\in\ZZ}
[\t\tau_n]^4\1_{\{\wt{T}_n\in(0,1]\}}\\ & = & {1\over
\alpha\lambda(\alpha)}\EE\sum_{n\in\ZZ}
\Bigl[{\tau_n(\alpha)\over\alpha}\Bigr]^4
\1_{\{T_n(\alpha)\in(0,\alpha]\}}\\ & = &
{1\over\alpha^4}\EE^0_\alpha[\tau_0]^4<\infty
\end{eqnarray*}
and for \ref{ass:spee}-\romii,
\begin{eqnarray*}
\wtEE^0[\wt{A}([\wt{R}^*_0,\wt{R}^*_1))]^4
& = & {1\over \t\lambda}\EE\sum_{n\in\ZZ}
    [\wt{A}([\wt{R}^*_-(\wt{T}_n),\wt{R}^*_+(\wt{T}_n)))]^4
    \1_{\{\wt{T}_n\in(0,1]\}}\\
& = & {1\over \alpha\lambda(\alpha)}\EE\sum_{n\in\ZZ}
    [A([R^*_-(T_n),R^*_+(T_n)))]^4
    \1_{\{T_n(\alpha^*)\in(0,\alpha^*]\}}\\
& = & \EE^0_{\alpha^*}[A([R_0^*,R_1^*))]^4<\infty.
\end{eqnarray*}

Finally, for \ref{ass:spee}-\romiii,
$$
\wtEE^0[f\bigl(\wt{W}_{\alpha^*}(0)\bigr)]^2
 =\EE^0_{\alpha^*}[f\bigl(W_{\alpha^*}(0)\bigr)]^2 <\infty.
$$

So we apply Theorem~\ref{thm:spee} and find for the right-hand derivative:
\begin{eqnarray*}
J'_r(\alpha)
& = & {\t\lambda\over\alpha^2}\wtEE^0\biggl\{
      \alpha\wt{W}'_\alpha(0)
      \Bigl[f(\wt{W}_\alpha(0))-f(\wt{W}_\alpha(\wt{T}_1-))\Bigr]\\
&   & \mbox{}+
       \Bigl[F(\wt{W}_\alpha(0))-F(\wt{W}_\alpha(\wt{T}_1-))\Bigl]\\
&   & \mbox{}-[\wt{W}_\alpha(0)-\wt{W}_\alpha(\wt{T}_1-)]
              f(\wt{W}_\alpha(\wt{T}_1-))\biggr\}.
\end{eqnarray*}
This gives Equation~(\ref{eq:arri:Jr}); the left-hand derivative is
derived similarly.
\end{proof}

\section{Implementation of the method.}
\label{sec:estimate}

The formulas given in preceding sections will be interesting only if
they provide estimates which are \romi\ easy to compute and \romii\
strongly consistent, which means that they converge a.s.\ to their
expected values. In this section we show how the estimate given by
Corollary~\ref{cor:princ} can be used in simulation when the system is
ergodic. In this case, ergodic theorem~(\ref{eq:ergopalm}) applied to
equation~(\ref{eq:cp:Jprime}) reads:
\begin{eqnarray}
J'(\theta) 
&=& \lim_{n\to\infty}{\lambda\over n}\sum_{k=0}^{n-1}
            W'_\theta(T_k)[f(W_\theta(T_k))-f(W_\theta(T_{k+1}-))]\nonumber\\
&=& \lim_{n\to\infty}{\lambda\over n}\sum_{k=0}^{n-1}
            W'_\theta(T_k)[f(W_\theta(T_k))-f(W_\theta(T_k-))].
                       \label{eq:esti:palm}
\end{eqnarray}

The different ingredients of this formula are easy to evaluate once
the simulation of the queue is set up: $W_\theta(T_k)$ and
$W_\theta(T_k-)$ are known when customer $k$ joins the queue; to get
$W'_\theta(T_k)$, we use equation (\ref{eq:lindley}), keeping in mind
that both $W(t)$ and $W'(t)$ are {\em c\`adl\`ag} processes and find:
$$
W'_\theta(T_k)=\cases{\sigma'_k(\theta)& if customer $k$ finds the
system empty\cr
W'_\theta(T_{k-1})+\sigma'_k(\theta)& else.}
$$

As shown in Section~\ref{sec:construction}, $\sigma'_k$ can most of
the time be expressed as a function of $\sigma_k$ and $\theta$, say
$\sigma'_k=D(\sigma_k,\theta)$. So if we define $w_k=W_\theta(T_k-)$
and $d_k=W'_\theta(T_k)$, we have
\begin{eqnarray*}
w_k&=&(w_{k-1}+\sigma_{k-1}-\tau_{k-1})^+\\
d_k&=&d_{k-1}\1{\{w_k>0\}}+D(\sigma_k,\theta),
\end{eqnarray*}
and equation~(\ref{eq:esti:palm}) shows that 
$$
\phi_n\egaldef{\lambda\over n}\sum_{k=0}^{n-1}
            d_k[f(w_k+\sigma_k)-f(w_k)]
$$
is a strongly consistent estimate of $J'(\theta)$. Since our Palm
estimate does not require differentiability for $f$, one will want to
check whether it is as accurate as the classic IPA estimate: if the
system is ergodic, the ergodic theorem~(\ref{eq:ergodic}) of the
appendix applied to equation~(\ref{eq:baseIPA}) gives
\begin{eqnarray}
J'(\theta)
& = & \lim_{t\to\infty}
          {1\over t}\int_0^t W_\theta'(s)f'(W_\theta(s))\,ds\nonumber\\
& = & \lim_{n\to\infty}
          {1\over T_n}\sum_{k=0}^{n-1}
           \int_{T_k}^{T_{k+1}}W_\theta'(s)f'(W_\theta(s))\,ds\nonumber\\
& = & \lim_{n\to\infty}
          {1\over T_n}\sum_{k=0}^{n-1}
            W'_\theta(T_k)[f(W_\theta(T_k))-f(W_\theta(T_{k+1}-))]
                                           \label{eq:esti:classic},
\end{eqnarray}
where all the limits are valid $\PP^0$-a.s.\ or $\PP$-a.s.\
indifferently. In the third equality, we used the fact that
$W'_\theta(t)$ is zero during idle periods. Comparing
equations~(\ref{eq:esti:palm}) and~(\ref{eq:esti:classic}), we see
that the estimates based on the same amount of data give very close
expressions; in fact, they are even equal when $\lambda$ needs to be
estimated. For comparisons between time-average and customer-average
estimates, see for instance Glynn and Whitt~\cite{GlyWhi:1}. The
implementation of an estimate of the second derivative of $J$ would be
done exactly in the same way, except that the formulas involved are
slightly more complicated.

\section*{Appendix: a short introduction to Palm theory.}\label{sec:palm}

In this appendix we give without proof some basic results on Palm
theory; interested readers can refer to Baccelli and
Br\'emaud~\cite{BacBre:1} for a more complete presentation of the
subject. The stationary framework is the following: given a
probability space $(\Omega,{\cal F},\PP)$, let $\theta_t$, $t\in\RR$
be a measurable flow $(\Omega,{\cal F})\mapsto(\Omega,{\cal F})$,
i.e.:
\begin{itemize}
\item $(t,\omega)\mapsto\theta_t \omega$ is measurable w.r.t.\ ${\cal
B}(\RR)\otimes{\cal F}$,
\item $\theta_t$ is bijective for all $t\in\RR$,
\item $\theta_t\circ\theta_s=\theta_{t+s}$ for all $t, s\in\RR$. In
particular, $\theta_0=$identity and $\theta_t^{-1}=\theta_{-t}$.
\end{itemize}

Note that there is nothing common between the flow $\theta_t$ and the
parameter $\theta$ of the queue; these are the traditional notations
in sensitivity analysis and Palm theory. We assume that
$\PP\circ\theta_t=\PP$. Let $\procZ{T_n}$ and $\procZ{U_n}$ be two
simple point process and let $A$ and $D$ be the associated counting
measures, that is, for all borelian set $C\in\RR$,
$$
A(C)=\sum_{n\in\ZZ}\1_C(T_n),\ D(C)=\sum_{n\in\ZZ}\1_C(U_n)
$$
and assume that for each $n\in\ZZ$, $U_n-T_n\egaldef W_n>0$. We take
the convention $T_0\leq0<T_1$ and note:
\begin{eqnarray*}
T_-(t)&=&\sup(T_n:T_n\leq t),\\
T_+(t)&=&\inf(T_n\ :\ T_n> t).
\end{eqnarray*}

$A$ and $D$ are viewed as arrival and departure processes and we note
$X(t)$ a queueing process associated with them. Let $B(t)$ be a non
decreasing {\em c\`adl\`ag}---i.e.\ right continuous with left
limits---real valued process and $Z(t)$ a non-negative real-valued
stochastic process. We assume that these processes are compatible with
the flow $\theta_t$, that is
\begin{eqnarray*}
A(\omega,C+t)&=&A(\theta_t\omega,C)\\
Z(\omega,t)&=&Z(\theta_t\omega,0).
\end{eqnarray*}

Similar equalities hold for $D$ and $X$; if we define
$\lambda_A=\EE[A((0,1])]$, then there exists a probability measure
called the Palm probability of the stationary process
$(A,\theta_t,\PP)$ verifying the Swiss Army Formula
(Br\'emaud~\cite{Bre:2}):
\begin{equation}\label{eq:swiss}
\lambda_A\EE_A^0[\int_{(0,W_0]}Z(s)dB(s)]={1\over t}\EE[\int_{(0,t]}X(s-)Z(s)dB(s)]
\end{equation}

The Swiss Army Formula is not the definition of the Palm measure, but
we will see that it contains this definition and the classic formulas
of Palm theory. We shall add that under $\PP_A^0$, $T_0=0$ a.s. We
derive now some useful formulas from ~(\ref{eq:swiss}). The first one
is the inversion formula: Take $U_n=T_{n+1}$ and $B(s)=s$; then
$X(t)=1$ and
\begin{equation}\label{eq:inversion}
\EE[Z(0)]=\lambda_A\EE_A^0[\int_0^{T_1}Z(s)ds].
\end{equation}

The second formula is Neveu's exchange formula (Neveu~\cite{Nev:1}):
we take $D$ as for the inversion formula and remark that if $B\equiv
A$, (\ref{eq:swiss}) reads
$$
\lambda_A\EE^0_A[Z(0)]={1\over t}\EE[\int_{(0,t]}Z(s)dA(s)].
$$
This is Mecke's definition of Palm probability. Now, if $B$ is a point
process, we use the above equation and (\ref{eq:swiss}) to derive the
exchange formula:
\begin{equation}\label{eq:exchange}
\lambda_A\EE^0_A[\int_0^{T_1}Z(s)dB(s)]=\lambda_B\EE^0_B[Z(0)].
\end{equation}

We will now prove a simple lemma which replaces Wald's identity for
stationary systems:

\begin{lem}\label{lem:sum}
Let $\procZ{R_n}$ be a stationary stochastic point process with
associated measure $B$. The following holds:
\begin{equation}\label{eq:lem:sum}
\EE_A^0\Bigl[\int_{[R_0,R_1)}Z(s)dA(s)\Bigr]=\EE_A^0[A([R_0,R_1))\,Z(0)].
\end{equation}
\end{lem}

\begin{proof}{}
If we note $Y$ the random variable inside the expectation of the
l.h.s.\ of equation~(\ref{eq:lem:sum}), then
$$
Y\circ \theta_{T_i}
=\int_{[R_-(T_i),R_+(T_i))}Z(s)dA(s)
=\sum_{n\in\ZZ}Z(T_n)\1_{[R_-(T_i),R_+(T_i))}(T_n).
$$

Since $R_\pm(T_i)=R_\pm(0)$ if $T_i\in[R_-(0),R_+(0))$,
\begin{eqnarray*}
\sum_{T_i\in[R_-(0),R_+(0))}Y\circ\theta_{T_i}
&=&\sum_{i\in\ZZ}\sum_{n\in\ZZ}
     Z(T_n)\1_{[R_-(0),R_+(0))}(T_n)\1_{[R_-(0),R_+(0))}(T_i)\\
&=&A([R_0,R_1))\sum_{n\in\ZZ}Z(T_n)\,\1_{[R_0,R_1)}(T_n).
\end{eqnarray*}
If we now apply Neveu's exchange formula~(\ref{eq:exchange}) between
$\PP^0_A$ and $\PP^0_B$, we obtain:
\begin{eqnarray*}
\lambda_A\EE_A^0Y
&=&\lambda_B\EE^0_B\sum_{T_i\in[R_-(0),R_+(0))}Y\circ\theta_{T_i}\\
&=&\lambda_B\EE^0_B\Bigl[\sum_{n\in\ZZ}
A([R_0,R_1))\,Z(T_n)\,\1_{[R_0,R_1)}(T_n)\Bigr]\\
&=&\lambda_A\EE^0_A[A([R_0,R_1))\,Z(0)],
\end{eqnarray*}
which is exactly equality~(\ref{eq:lem:sum}).
\end{proof}

Formula~(\ref{eq:lem:sum}) can be seen as an extension of Wald's
identity that can be used for stationary sequences instead of i.i.d.\
variables and applies to any stationary process. It is not as
convenient as Wald's identity is, but is valid in a wider setting.

\begin{remark}
Lemma~\ref{lem:sum} is also a corollary of the extended $H=\lambda G$
formula (6.2) of Br\'emaud~\cite{Bre:2}.
\end{remark}

Palm probabilities can also be given an interpretation which relates
them to simulation. When $(P,\theta_t)$ is ergodic, the ergodic
theorems for $\PP$ and $\PP^0$ read:
\begin{eqnarray*}
\EE[Y]&=& \lim_{t\to\infty}{1\over t}\int_0^tY\circ\theta_sds\\
\EE^0_A[Y]&=&\lim_{n\to\infty}{1\over n}\sum_{k=0}^{n-1}Y\circ\theta_{T_k}
\end{eqnarray*}
which imply that:
\begin{eqnarray}
\EE[Z(0)]&=& \lim_{t\to\infty}{1\over t}\int_0^tZ(s)ds\label{eq:ergodic}\\
\EE^0_A[Z(0)]&=&\lim_{n\to\infty}{1\over n}\sum_{k=0}^{n-1}Z(T_k).\label{eq:ergopalm}
\end{eqnarray}

These equalities are valid $\PP$-a.s.\ and $\PP^0_A$-a.s. This shows
that $\EE[Z(0)]$ is the {\em time-average} of the process $Z(t)$,
whether $\EE^0_A[Z(0)]$ is its {\em customer-average}.

\bibliography{queues}
\bibliographystyle{acm}

\end{document}